\documentclass[reqno, oneside, 12pt]{amsart}

\usepackage[letterpaper]{geometry}
\usepackage{amsmath}
\usepackage{amssymb}
\usepackage{enumerate}
\usepackage{enumitem}
\usepackage{mathrsfs}
\usepackage{mathtools}
\usepackage{stmaryrd}

\mathtoolsset{centercolon}

\numberwithin{equation}{section}

\newtheorem{thm}{Theorem}[section]
\newtheorem{cor}[thm]{Corollary}
\newtheorem{lem}[thm]{Lemma}
\newtheorem{prop}[thm]{Proposition}

\theoremstyle{remark}
\newtheorem*{rem*}{Remark}
\newtheorem{example}{Example}

\theoremstyle{definition}

\renewcommand{\bar}{\overline}

\newcommand{\Z}{\mathbb{Z}}

\newcommand{\N}{\mathbb{N}}

\newcommand{\SL}{\textup{SL}}

\renewcommand{\S}{\mathcal{S}}

\newcommand{\legendre}[2]{\pfrac{#1}{#2}}

\renewcommand{\pmatrix}[4]{\left( \begin{smallmatrix} #1 & #2 \\ #3 & #4 \end{smallmatrix} \right)}
\newcommand{\pMatrix}[4]{\left( \begin{matrix} #1 & #2 \\ #3 & #4 \end{matrix} \right)}

\newcommand{\pfrac}[2]{\left(\frac{#1}{#2}\right)}

\date{}
\author{Nickolas Andersen}
\address{Department of Mathematics\\
University of Illinois\\
Urbana, IL 61801} 
\email{nandrsn4@illinois.edu}

\title[Classifying Mock Theta Congruences]{Classification of Congruences for Mock Theta Functions and Weakly Holomorphic Modular Forms}

\begin{document}

\begin{abstract}
Let $f(q)$ denote Ramanujan's mock theta function
\[
	f(q) = \sum_{n=0}^{\infty} a(n) q^{n} := 1 + \sum_{n=1}^{\infty} \frac{q^{n^{2}}}{(1+q)^{2}(1+q^{2})^{2} \cdots (1+q^{n})^{2}}.
\]
It is known that there are many linear congruences for the coefficients of $f(q)$ and other mock theta functions. We prove that if the linear congruence $a(mn+t) \equiv 0 \pmod{\ell}$ holds for some prime $\ell \geq 5$, then $\ell | m$ and $\legendre{24t-1}{\ell} \neq \legendre{-1}{\ell}$. We prove analogous results for the mock theta function $\omega(q)$ and for a large class of weakly holomorphic modular forms which includes $\eta$-quotients. This extends work of Radu \cite{Radu:ao} in which he proves a conjecture of Ahlgren and Ono for the partition function $p(n)$.
\end{abstract}

\maketitle
\thispagestyle{empty}

\section{Introduction and Statement of Results}

In his famous last letter to Hardy, Ramanujan introduced 17 strange $q$-series which he called \emph{mock theta functions}, many of whose coefficients encode interesting combinatorial information. A prototypical example is the function
\begin{equation}
	f(q) = \sum_{n=0}^{\infty} a(n) q^{n} := 1 + \sum_{n=1}^{\infty} \frac{q^{n^{2}}}{(1+q)^{2}(1+q^{2})^{2}\cdots (1+q^{n})^{2}},
\end{equation}
whose coefficients are related to partition ranks (see \cite{AL}). The function $f(q)$ resembles the generating function for partitions in the following way: by separating a partition into its Durfee square and two partitions into parts of size $\leq n$, we have
\[
	\sum_{n=0}^{\infty} p(n) q^{n} = 1 + \sum_{n=1}^{\infty} \frac{q^{n^{2}}}{(1-q)^{2}(1-q^{2})^{2}\cdots (1-q^{n})^{2}}.
\]
The latter generating function has many interesting arithmetic properties arising from its relation to the modular form
\[
	\eta(z) := q^{1/24} \prod_{n=1}^{\infty} (1-q^{n}), \qquad q:=e^{2\pi i z},
\]
via the well-known relation
\[
	\frac{1}{\eta(z)} = \sum_{n=0}^{\infty} p(n) q^{n-1/24}.
\]

The arithmetic properties of $p(n)$ have been the subject of a vast amount of research, much of which has focused on congruence properties satisfied by $p(n)$ (for a few examples, see \cite{AB, Atkin:conj, Atkin, BW, FKO, LO}). We make particular mention of work of Ahlgren and Ono \cite{Ahlgren:dist, AO:conj, Ono:dist} which shows the existence of congruences
\[
	p(mn+t) \equiv 0 \pmod{\ell^{j}}
\]
for all prime powers $\ell^{j}$ with $\ell\geq 5$. In their systematic theory, the progressions $\{mn+t\}$ are subprogressions of $\{\ell n + \beta\}$, where $\beta$ lies in one of $(\ell+1)/2$ classes modulo $\ell$. In \cite{AO:conj}, they conjectured that congruences do not exist for $p(n)$ outside these progressions, and recently Radu \cite{Radu:ao} proved their conjecture. In particular, he showed that if
\[
	p(mn+t) \equiv 0 \pmod{\ell} 
\]
for some prime $\ell \geq 5$, then $\ell | m$ and $\legendre{24t-1}{\ell} \neq \legendre{-1}{\ell}$. 

Many authors have studied congruence properties of mock theta functions (see, for example, \cite{Alfes, BO:omega, BOPR, Chan, Waldherr}). Recent work of Fuller, Friedlander, Goodson, and the author \cite{AFFG} shows the existence of linear congruences for all of Ramanujan's mock theta functions. As an example, for any prime $\ell \geq 5$, the coefficients of $f(q)$ satisfy
\[
	a(mn+t) \equiv 0 \pmod{\ell}
\]
for infinitely many progressions $\{mn+t\}$. As with $p(n)$, in each of these cases $\ell | m$ and $\legendre{24t-1}{\ell} \neq \legendre{-1}{\ell}$. It is natural to ask whether the analogue of the conjecture in \cite{AO:prop} holds in this case as well. Here we will prove

\begin{thm} \label{thm:f-q}
	Let $\ell \geq 5$ be prime. If for all $n$,
	\[
		a(mn+t) \equiv 0 \pmod{\ell},
	\]
	then $\ell | m$ and $\legendre{24t-1}{\ell} \neq \legendre{-1}{\ell}$. 
\end{thm}

All known congruences for $f(q)$ are constructed by applying results from the theory of modular forms (see, for instance \cite{Treneer:cong, Treneer:quad}) to certain twists of $f(q)$ which are known to be weakly holomorphic modular forms. Theorem \ref{thm:f-q} implies that all linear congruences $a(mn+t)\equiv 0 \pmod{\ell}$ with $\ell \nmid (24t-1)$ arise in this way.

\begin{cor} \label{cor:f-q}
Suppose that $a(mn+t) \equiv 0 \pmod{\ell}$ for all $n$, and that $\ell \nmid (24t-1)$. Then the function
\[
	q^\frac{t-1/24}{m}\sum a(mn+t) q^{n}
\]
is a weakly holomorphic modular form.
\end{cor}

We will prove analogous statements for the mock theta function
\begin{equation}
	\omega(q) = \sum_{n=0}^{\infty} c(n) q^{n} := 1 + \sum_{n=1}^{\infty} \frac{q^{2n^{2}+2n}}{(1+q)^{2}(1+q^{3})^{2}\cdots(1+q^{2n-1})^{2}},
\end{equation}
which appears naturally with $f(q)$ as a component of a vector-valued mock modular form (see, for example, \cite{Zwegers}).

\begin{thm} \label{thm:omega-q}
Let $\ell \geq 5$ be prime. If for all $n$,
\[
	c(mn+t) \equiv 0 \pmod{\ell},
\]
then $\ell | m$ and $\legendre{3t+2}{\ell} \neq \legendre{-1}{\ell}$. 
Furthermore, if $\ell \nmid (3t+2)$ then the function
\[
	q^\frac{t+2/3}{m}\sum c(mn+t) q^{n}
\] 
is a weakly holomorphic modular form.
\end{thm}

It is natural to ask whether these results extend to a larger class of modular forms. Treneer \cite{Treneer:cong, Treneer:quad} has extended the results of \cite{AO:prop} to cover any weakly holomorphic modular form. In Section \ref{sec:weak-hol} we prove a general theorem which applies to the class of $\eta$-quotients, which appear as generating functions of many combinatorial objects. Using the standard notation
\[
	f(z) = \prod_{\delta|N} \eta(\delta z)^{r_{\delta}},
\]
we define $B := \sum \delta r_{\delta}$. If $m$ is a positive integer, write $m=2^u3^vm'$ and $B=2^r3^sB'$ with $(m',6)=(B',6)=1$ and define a divisor $m_B$ of $m$ by
\begin{equation} \label{eq:def-m-B}
	m_B := 2^{\min(r,u)} 3^{\min(s,v)} m'.
\end{equation}
The following theorem is an analgoue of Theorems \ref{thm:f-q} and \ref{thm:omega-q} for $\eta$-quotients.

\begin{thm} \label{thm:eta}
Let $\ell\geq 5$ be prime. Suppose that $f(z) = q^{B/24} \sum a_{f}(n) q^{n}$ is an $\eta$-quotient as above with $B<0$ and $\ell \nmid BN$. If $m$ and $t$ are positive integers with $(m_B,N)=1$ and
\[
	a_{f}(mn+t) \equiv 0 \pmod{\ell}
\]
for all $n$,
then $\ell | m$ and $\legendre{24t+B}{\ell} \neq \legendre{B}{\ell}$.
\end{thm}

We give some applications of Theorem~\ref{thm:eta}.

\begin{example}
	Let $p_{k}(n)$ denote the number of partitions of $n$ into $k$ colors. Then we have the generating function
\[
	\sum_{n=0}^{\infty} p_{k}(n) q^{n-k/24} = \eta^{-k}(z).
\]
In \cite{Andrews:multi}, Andrews showed that if $\ell \geq 5$ is prime and $\legendre{8t+1}{\ell}\neq 1$ then
\begin{equation} \label{eq:k-color-cong}
	p_{\ell-3} (\ell n + t) \equiv 0 \pmod{\ell}.
\end{equation}
Applying Theorem $\ref{thm:eta}$ with $B=3-\ell$ we obtain the necessary condition $\legendre{24t+3-\ell}{\ell} \neq \legendre{3-\ell}{\ell}$, which is equivalent to $\legendre{8t+1}{\ell} \neq 1$. So we have
\begin{cor}
	Let $\ell \geq 5$ be prime. Then
\[
	p_{\ell-3}(\ell n+t) \equiv 0 \pmod{\ell}
\]
if and only if $\legendre{8t+1}{\ell} \neq 1$.
\end{cor}
\end{example}

\begin{example}
	Broken $1$-diamond partitions are defined by the generating function
\[
	\sum_{n=0}^{\infty} \Delta_{1}(n) q^{n-1/6} = \frac{\eta(2z)\eta(3z)}{\eta^{3}(z)\eta(6z)}.
\]
In \cite{Mortenson}, Mortenson showed that if $\ell$ is a prime satisfying
\[
	\ell \equiv 1, 25, 37, 47, 59, \text{ or }83 \pmod{84}
\]
then there exist infinitely many progressions $\{mn+t\} \subseteq \{\ell n + \beta\}$ for each $\beta$ satisfying $\legendre{6\beta-1}{\ell} \neq \legendre{-1}{\ell}$ such that
\[
	\Delta_{1}(mn+t) \equiv 0 \pmod{\ell}.
\]
Theorem~\ref{thm:eta} shows that any such congruence with $m$ odd must satisfy $\{mn+t\}\subseteq \{\ell n+\beta\}$ for some $\beta$ satisfying $\legendre{6\beta-1}{\ell} \neq \legendre{-1}{\ell}$.
\end{example}

\begin{example}
The Andrews-Stanley partition function $t(n)$ (see \cite{Andrews:stanley, Stanley:prob, Stanley}) is defined as the number of partitions $\pi$ having the property that the number of odd parts of $\pi$ is congruent to the number of odd parts of the conjugate partition $\pi'$ modulo $4$. The generating function for $t(n)$ is
\[
	\sum_{n=0}^{\infty} t(n) q^{n-1/24} = \frac{\eta^{2}(2z)\eta^{5}(16z)}{\eta(z)\eta^{5}(4z)\eta^{2}(32z)}.
\]
In \cite{Swisher}, Swisher proved that for each prime $\ell \geq 5$ there are infinitely many progressions $\{mn+t\}$ for which $t(n)$ and $p(n)$ satisfy the simultaneous congruences
\[
	t(mn+t) \equiv p(mn+t) \equiv 0 \pmod{\ell}.
\]
Theorem~\ref{thm:eta} shows that any linear congruence
\[
	t(mn+t) \equiv 0 \pmod{\ell}
\]
must satisfy $\ell | m$ and $\legendre{24t-1}{\ell} \neq \legendre{-1}{\ell}$, mirroring the $p(n)$ case.
\end{example}

\begin{example}
	The assumption $\ell \nmid NB$ is necessary in general. Define a sequence $a_{\ell}(n)$ of integers modulo $\ell$ by the relation
	\[
		\sum_{n=0}^{\infty} a_{\ell}(n) q^{n-\ell/24} := \eta^{-1}(\ell z) \equiv \eta^{-\ell}(z) \pmod{\ell}.
	\]
	Since $\eta^{-1}(\ell z)$ and $\eta^{-\ell}(z) \pmod{\ell}$ are supported only on exponents divisible by $\ell$, the sequence $a_{\ell}(n)$ satisfies $a_{\ell}(\ell n + t)\equiv 0 \pmod{\ell}$ for all $t$ coprime to $\ell$.

	The assumption $(m_B,N)=1$ is also necessary. We give examples illustrating two cases when $(m,N)>1$. Let $\delta>1$ and define
	\[
		\sum_{n=0}^\infty b_\delta(n) q^{n-\delta/24} := \eta^{-1}(\delta z).
	\]
	Then $b_{\delta}(n) = 0$ unless $\delta | n$, so $b_{\delta}(\delta n + t)=0$ for all $t$ not divisible by $\delta$. Here $B=-\delta$ and $N=\delta$, so $(m_B,N)=(m,N)=\delta$. Now consider the sequence $c_3(n)$ defined by
	\[
		\sum_{n=0}^\infty c_3(n) q^{n-1/6} := \eta^{-1}(z) \eta^{-1}(3z).
	\]
	The sequence $c_3(n)$ does not vanish trivially on any arithmetic progression as does $b_\delta(n)$. Here $B=-4$ and $N=3$, so $(m_B,N)=1$. In this case $m$ is allowed to be a multiple of $3$.
\end{example}

\section{Preliminaries} \label{sec:preliminaries}

Ramanujan's mock theta functions are examples of weight $1/2$ mock modular forms, which are the holomorphic parts of harmonic Maass forms (see Sections 6 and 7 of \cite{Ono:master} for definitions and details). Each harmonic Maass form $F$ decomposes uniquely as $F = f + NH$, where $f$ is the holomorphic part (or mock modular form) and $NH$ is the non-holomorphic part. 

Define
\begin{gather*}
	g_{0}(z) := \sum_{n\in \Z} (-1)^n (n+1/3) \, q^{\frac{3}{2}\left(n+\frac{1}{3}\right)^2}, \\
	g_1(z) := - \sum_{n\in\Z} (n+1/6) \, q^{\frac{3}{2}\left(n+\frac{1}{6}\right)^2}.
\end{gather*}
Then the functions
\begin{gather}
	M(z) := q^{-1/24} f(q) - 2i\sqrt{3} \int_{-\bar{z}}^{i\infty} \frac{g_1(\tau)}{(-i(z+\tau))^\frac{1}{2}} d\tau
\end{gather}
and
\begin{gather}
	\Omega(z) := 2q^{2/3} \omega(q) - 2i\sqrt{3} \int_{-\bar{2z}}^{i\infty} \frac{g_0(\tau)}{(-i(2z+\tau))^\frac{1}{2}} d\tau
\end{gather}
are the completed harmonic Maass forms associated to $f(q)$ and $\omega(q)$, respectively. These functions satisfy the following transformation laws (see \cite[Theorems 2.1--2.4]{Andrews:WD}, \cite[Proof of Corollary 2.3]{BO}, \cite[(3.13) and (4.2)]{AK}). For $A =\pmatrix{a}{b}{c}{d} \in \Gamma_{0}(2)$ with $c>0$ we have
\begin{equation} \label{eq:M-transform}
	M \pfrac{az+b}{cz+d} = w(A) \ (cz+d)^{1/2} M(z),
\end{equation}
where $w(A)$ is the $24$th root of unity given by
\begin{equation}
	w(A) := i^{-\frac{1}{2}} (-1)^{\frac{c+1+ad}{2}} \exp\left( 2\pi i \left(-\frac{1}{2}s(-d,c) - \frac{a+d}{24c} - \frac{a}{4} + \frac{3cd}{8} \right) \right)
\end{equation}
and $s(d,c)$ is the Dedekind sum defined by
\begin{equation} \label{eq:def-ded-sum}
	s(d,c) := \sum_{r=1}^{c-1} \left( \frac{r}{c} - \left\lfloor\frac{r}{c}\right\rfloor - \frac{1}{2} \right)  \left( \frac{dr}{c} - \left\lfloor\frac{dr}{c}\right\rfloor - \frac{1}{2} \right).
\end{equation}
For $A = \pmatrix{a}{b}{c}{d} \in \SL_{2}(\Z)$ with $c>0$ we have
\begin{equation} \label{eq:O-transform}
	\Omega(Az) = 
	\begin{cases}
		w_{1}(A) (cz+d)^{1/2} \Omega(z) & \text{if $c$ is even,} \\
		w_{2}(A) 2^{-1/2} (cz+d)^{1/2} M(z/2) & \text{if $d$ is even,}
	\end{cases}
\end{equation}
where $w_{1}(A)$ and $w_{2}(A)$ are roots of unity defined by
\begin{equation}
\begin{aligned}
	w_{1}(A) &:= i^{1/2} \exp\left( 2\pi i \left(\frac{a-1}{4} -\frac{s(-d,c/2)}{2} + \frac{3ab}{4} - \frac{a+d}{12c} \right)\right),  \\
	w_{2}(A) &:= (-i)^{1/2} \exp\left( 2\pi i \left( \frac{32a-d}{48c} -\frac{s(-d/2,c)}{2} - \frac{2a+b-3-3ab+3a/c}{4} \right)\right).
\end{aligned}
\end{equation}

The Dedekind sum \eqref{eq:def-ded-sum} satisfies the following transformation law which is an easy consequence of \cite[Lemma 2]{lewis}.

\begin{lem} \label{lem:dedekind-sum}
Let $m,\lambda$ be positive integers with $(m,6)=1$. Then for every $\pmatrix{a}{b}{c}{d} \in \Gamma_{0}(m)$ with $c>0$ and $(a,6)=1$ we have
\begin{equation} \label{eq:s(d,c)-transformation}
	s(-d + \lambda c, mc) = s(-d,mc) + \lambda \frac{1-a^{2}}{12m} + \text{an even integer}.
\end{equation}
\end{lem}

\begin{proof}
	From \cite[Lemma 2(i)]{lewis} we have (correcting a sign error),
\[
	s(d+c,mc) = s(d,mc) + \frac{1-a^{2}}{12m} + \text{an even integer}.
\]
Applying this iteratively to the matrices
\[
	\pMatrix{-a}{b-ja}{c}{-d+jc} \in \Gamma_{0}(m), \quad 0\leq j\leq \lambda-1,
\]
we obtain \eqref{eq:s(d,c)-transformation}.
\end{proof}

We recall the transformation law for $\eta(z)$ (see \cite[Chapter 4, Theorem 2]{Knopp}). If $A=\pmatrix{a}{b}{c}{d} \in \SL_{2}(\Z)$ with $c>0$ we have
\begin{equation} \label{eq:eta-transformation}
	\eta\pfrac{az+b}{cz+d} = \xi(A) (cz+d)^{1/2} \, \eta(z),
\end{equation}
where
\begin{equation} \label{eq:eta-mult}
	\xi(A) = 
	\begin{cases}
		\legendre{c}{|d|} \exp\left( \frac{2\pi i}{24}\left[ (a+d)c - bd(c^{2}-1) + 3d - 3 - 3cd \right] \right) & \text{ if $c$ is even,} \\
		\legendre{d}{|c|} \exp\left( \frac{2\pi i}{24}\left[ (a+d)c - bd(c^{2}-1) -3c \right] \right) & \text{ if $c$ is odd.}
	\end{cases}
\end{equation}

The proofs below adapt Radu's methods in \cite{Radu:ao, Radu:sub} and require the following technical lemma which is a generalization of \cite[Theorem 4.2]{Radu:sub}. Given an integer $B=2^r3^sB'$ and a positive integer $m=2^u3^vm'$ with $(B',6)=(m',6)=1$, define $Q_{m,B}:=2^\alpha3^\beta m'$, where
\begin{equation} \label{eq:def-alpha-beta}
	\alpha := 
	\begin{cases}
		0 &\text{ if } r=0,\\
		\min(r,u) &\text{ if } r=1,2,\\
		u &\text{ if } r\geq 3,
	\end{cases}
	\qquad \qquad
	\beta :=
	\begin{cases}
		0 &\text{ if }s=0,\\
		v &\text{ if }s\geq 1.
	\end{cases}
\end{equation}
Note that the primes dividing $Q_{m,B}$ are the same as those dividing $m_B$ defined in \eqref{eq:def-m-B}.

\begin{lem} \label{lem:radu}
Let $m,t,B,N\in \Z$ with $m,N>0$ and write $m=2^u3^vm'$ with $(m',6)=1$. Let $\alpha$, $\beta$, and $Q_{m,B}$ be as above. Suppose that $\lambda$ is an integer with $0\leq \lambda < m/Q_{m,B}$. Then there exists an integer $a=a_\lambda$ with $(a,6mN)=1$ and
\begin{equation} \label{eq:radu-lem}
	t+\lambda Q_{m,B} \equiv ta^2 + B\frac{a^2-1}{24} \pmod{m}.
\end{equation}
\end{lem}

\begin{proof} Write $Q=Q_{m,B}$ for convenience.
We will construct integers $b_0, \ldots , b_{u-\alpha}$ with $(b_n,6mN)=1$ and
\[
	t + \lambda Q \equiv t b_n^2 +  B \frac{b_n^2-1}{24} \pmod{2^n Q},
\]
and integers $c_0, \ldots, c_{v-\beta}$ with $(c_n,6mN)=1$ and
\[
	t + \lambda Q \equiv t c_n^2 +  B \frac{c_n^2-1}{24} \pmod{2^{u-\alpha} 3^n Q}.
\]
Then \eqref{eq:radu-lem} holds with $a=c_{v-\beta}$ since $m=2^{u-\alpha}3^{v-\beta}Q$.

Let $N'$ denote the largest divisor of $N$ coprime to $6$. We begin by constructing the integers $b_n$. Let $b_0:=1$. If $\alpha=u$, we are done, so assume $\alpha<u$. Then either $\alpha=r=0$ or $\alpha=\min(r,u)=r$. In either case, $2^\alpha || B$. For $n\geq 1$, let $x_{n-1}$ be the unique integer satisfying
\[
	B\frac{b_{n-1}^2-1}{24} + t(b_{n-1}^2-1) - \lambda Q = x_{n-1}2^{n-1}Q
\]
and define $b_n := b_{n-1} + 2^{n+1-\alpha} 3 x_{n-1}QN'$.
Then since $2^\alpha | Q$, we have
\[
	B(b_n^2-1) \equiv B(b_{n-1}^2-1) + 2^{n+2-\alpha}3x_{n-1}QBN'b_{n-1} \pmod{24\cdot 2^nQ}.
\]
Therefore
\begin{align*}
	B\frac{b_n^2-1}{24}& + t(b_n^2-1) - \lambda Q \\
		&\equiv B\frac{b_{n-1}^2-1}{24} + t(b_{n-1}^2-1) - \lambda Q + 2^{n-1-\alpha}x_{n-1}QBN'b_{n-1} &\pmod{2^n Q} \\
		&\equiv x_{n-1} 2^{n-1} Q(1+2^{-\alpha}BN'b_{n-1}) &\pmod{2^n Q} \\
		&\equiv 0 &\pmod{2^n Q}
\end{align*}
since $2^{-\alpha}BN'b_{n-1}$ is odd.

To construct the integers $c_n$, set $c_0:=b_{u-\alpha}$. If $\beta=v$, we are done. If $\beta<v$ then $\beta=0$ and $3\nmid B$. In this case let $y_{n-1}$ be the unique integer satisfying
\[
	B \frac{c_{n-1}^2-1}{24} + t(c_{n-1}^2-1) - \lambda Q = y_{n-1} 2^{u-\alpha} 3^{n-1} Q
\]
and define $c_n:= c_{n-1}(2^{2+u-\alpha}3^ny_{n-1}QN'-\epsilon)$, where $\epsilon \in \{-1,1\}$ and $\epsilon \equiv BN' \pmod{3}$. Then
\begin{align*}
	B \frac{c_n^2-1}{24}& + t(c_n^2-1) - \lambda Q \\
		&\equiv B\frac{c_{n-1}^2-1}{24} + t(c_{n-1}^2-1) - \lambda Q - \epsilon2^{u-\alpha} 3^{n-1} y_{n-1} QBN'c_{n-1}^2 &\pmod{2^{u-\alpha}3^nQ} \\
		&\equiv 2^{u-\alpha}3^{n-1}y_{n-1}Q(1-\epsilon BN'c_{n-1}^2) &\pmod{2^{u-\alpha}3^nQ} \\
		&\equiv 0 &\pmod{2^{u-\alpha}3^nQ}
\end{align*}
since $\epsilon BN'c_{n-1}^2 \equiv \epsilon^2 \equiv 1 \pmod{3}$.
\end{proof}

We require two theorems of Deligne and Rapoport \cite{DR}, which relate the expansion of a modular form at $\infty$ with its expansions at other cusps. For a prime ideal $\pi$ and a modular form $f$ with $\pi$-integral coefficients, let $v_{\pi}(f)$ denote the $\pi$-adic valuation of $f$. Let $\zeta_m := \exp(2\pi i/m)$.

\begin{thm}\cite[VII, Cor 3.12]{DR} \label{thm:deligne-mod-p}
Let $f \in M_{k}(\Gamma(N)) \cap \Z[\zeta_{N}]\llbracket q \rrbracket$, $p$ a prime, and $\gamma \in \Gamma_{0}(p^{m})$, where $p^{m} || N$. Let $\pi$ be a prime of $\Z[\zeta_{N}]$ lying above $p$. Then
\[
	v_{\pi} (f) = v_{\pi} (f|_{k}\gamma).
\]
\end{thm}

\begin{thm}\cite[VII, Cor 3.13]{DR} \label{thm:deligne-coeff}
Let $f\in M_{k}(\Gamma(N)) \cap \Z[1/N,\zeta_{N}]\llbracket q \rrbracket$, $\gamma \in \SL_{2}(\Z)$. Then $f|_{k} \gamma \in \Z[1/N,\zeta_{N}]\llbracket q \rrbracket$.
\end{thm}

\section{Proof of Theorem~\ref{thm:f-q}}

Define
\begin{equation} \label{eq:def-M-m-t}
	M_{m,t}(z) :=  \frac{1}{m} \sum_{\lambda=0}^{m-1} \zeta_{m}^{-\lambda(t-1/24)} M \pfrac{z+\lambda}{m}.
\end{equation}
In order to work with the function $M_{m,t}(z)$, we need to isolate progressions on which the non-holomorphic part of $M_{m,t}(z)$ vanishes. We call a progression $t\pmod{m}$ \emph{good} if there exists a prime $p|m$ with $(\frac{1-24t}{p}) = -1$. By \cite{AK}, if $t\pmod{m}$ is good, then $M_{m,t}$ is a weakly holomorphic modular form. Thus Corollary \ref{cor:f-q} follows immediately from Theorem \ref{thm:f-q} and \eqref{eq:M-m-t-exp} below.

Suppose $a(mn+t) \equiv 0 \pmod{\ell}$ for some progression $t \pmod{m}$ that is not good. Choose a prime $p\geq 5$ with $p\nmid \ell m$ and a quadratic non-residue $x \pmod{p}$, and choose $T$ satisfying
\begin{align*}
	T &\equiv t \pmod{m}, \\
	T &\equiv \frac{1-x}{24} \pmod{p}.
\end{align*}
Then the subprogression $T \pmod{mp}$ is good. Note that the statements $\ell | mp$ and $\legendre{24T-1}{\ell} \neq \legendre{-1}{\ell}$ together imply that $\ell | m$ and $\legendre{24t-1}{\ell} \neq \legendre{-1}{\ell}$. Therefore to prove Theorem~\ref{thm:f-q} we may assume that $t \pmod{m}$ is a good progression. A calculation shows that if $t\pmod{m}$ is good, then
\begin{equation} \label{eq:M-m-t-exp}
	M_{m,t}(z) = q^{\frac{t-1/24}{m}} \sum a(mn+t)q^{n}.
\end{equation}

We will need a proposition from \cite{AK} which describes the transformation of $M_{m,t}$ under the action of $\Gamma_{0}(N_{m})$, where
\[
	N_{m} :=
	\begin{cases}
		2m & \text{if $(m,6)=1$}, \\
		8m & \text{if $(m,6)=2$}, \\
		6m & \text{if $(m,6)=3$}, \\
		24m & \text{if $(m,6)=6$}. \\
	\end{cases}
\]
If $A = \pmatrix{a}{b}{c}{d} \in \Gamma_{0}(N_{m})$ has $3\nmid a$ then we define $t_{A}$ to be any integer satisfying
\begin{equation} \label{def:t-A}
	t_{A} \equiv ta^{2} + \frac{1-a^{2}}{24} \pmod{m}.
\end{equation}

\begin{prop} \label{prop:M-m-t-M-m-t-A}
	For every $A = \pmatrix{a}{b}{c}{d} \in \Gamma_{0}(N_{m})$ with $3\nmid a$ we have
	\[
		M_{m,t}^{24m} \big|_{12m} A = M_{m,t_{A}}^{24m}.
	\]
\end{prop}

The matrices in $\Gamma_{0}(N_{m})$ with $3\nmid a$ generate $\Gamma_{1}(N_{m})$. Therefore we have
\begin{prop} \label{prop:M-m-t-modular}
	Suppose that $t \pmod m$ is good. Then there exists $s \in \N$ such that
	\begin{equation} \label{eq:delta-M-m-t-modular}
		\Delta^{s} M_{m,t}^{24m} \in M_{12(m+s)}(\Gamma_{1}(N_{m})).
	\end{equation}
\end{prop}

A useful consequence of these propositions is the following lemma.

\begin{lem} \label{lem:M-t-t-A}
	Suppose that $t\pmod{m}$ is good. Let $A=\pmatrix{a}{b}{c}{d} \in \Gamma_{0}(N_{m})$ with $3\nmid a$ and let $t_A$ be as in \eqref{def:t-A}. If
		$M_{m,t} \equiv 0 \pmod{\ell}$
	then
		$M_{m,t_A} \equiv 0 \pmod{\ell}$.
\end{lem}

\begin{proof}
Let $s$ be as in \eqref{eq:delta-M-m-t-modular}. Then by Proposition~\ref{prop:M-m-t-M-m-t-A} we have
	\[
		\Delta^{s} M_{m,t}^{24m} \big|_{12(m+s)} A = \Delta^{s }M_{m,t_{A}}^{24m}.
	\]
Applying Theorem~\ref{thm:deligne-mod-p} we conclude that
\[
	M_{m,t_A} \equiv 0 \pmod{\ell}. \qedhere
\]
\end{proof}

The following proposition proves the assertion in Theorem~\ref{thm:f-q} that $\ell | m$ and reduces us to the case $m=Q\ell$, where $(Q,6\ell)=1$.

\begin{prop} \label{prop:M-Q-ell}
	Suppose that $t\pmod{m}$ is good and write $m=2^u 3^v \ell^j Q$ with $(Q,6\ell)=1$. If $M_{m,t} \equiv 0 \pmod{\ell}$ then $j > 0$. If additionally $\ell \nmid (24t-1)$ then $M_{Q\ell,t}\equiv 0 \pmod{\ell}$.
\end{prop}

\begin{proof}
We begin by proving that $M_{Q\ell^j,t}\equiv 0 \pmod{\ell}$. Applying Lemma~\ref{lem:radu} with $B=-1$ and $N=1$ we see that as $a$ ranges over integers with $(a,6m)=1$, the quantity $t_{A}$ covers each of the progressions $t+\lambda Q\ell^j \pmod m$ for $0\leq \lambda  < 2^{u}3^{v}$. This fact and Lemma~\ref{lem:M-t-t-A} together imply that
\[
	M_{m,t+\lambda Q\ell^j} \equiv 0 \pmod{\ell}
\]
for $0\leq \lambda < 2^{u}3^{v}$. We have
\[
	\bigcup_{\lambda =0}^{2^{u}3^{v}-1}\{mn+(t+\lambda Q\ell^j)\} = \bigcup_{\lambda =0}^{2^{u}3^{v}-1}\{Q\ell^j(2^{u}3^{v}n+\lambda ) +t\} = \{Q\ell^j n+t\},
\]
since every integer can be written as $2^{u}3^{v}n+\lambda $ for some $\lambda $ with $0\leq \lambda  < 2^{u}3^{v}$. Therefore $M_{Q\ell^j,t} \equiv 0 \pmod{\ell}$.

We will now show that $j>0$. Suppose, by way of contradiction, that $j=0$, so that $M_{Q\ell^j,t} = M_{Q,t}$. For a sufficiently large $s\in \N$, define
\[
	g:= \Delta^{s} M_{Q,t}^{24Q} \in M_{12(Q+s)}(\Gamma_{1}(2Q)).
\]
Then by \cite[Proposition 9]{AK}, the leading term in the $q$-expansion of $g$ at the cusp $1/2$ is given by
\[
	g \big|_{12(Q+s)} \pMatrix{1}{0}{2}{1} = Q^{-12Q} q^{-Q^{2}+s} + \cdots.
\]
By assumption, $M_{Q,t} \equiv 0 \pmod{\ell}$, so the modular form $\ell^{-24Q}g$ has integral coefficients. But by Theorem~\ref{thm:deligne-coeff}, the coefficients of 
\[
	\ell^{-24Q} g \big|_{12(Q+s)} \pMatrix{1}{0}{2}{1}
\]
lie in $\Z[1/2Q,\zeta_{2Q}]$, which contradicts $\ell \nmid Q$.

We now show that $M_{Q\ell,t} \equiv 0 \pmod{\ell}$ under the assumption $\ell \nmid (24t-1)$. By \cite[Lemma 4.6]{Radu:ao}, for each $r$ with $0\leq r < \ell^{j-1}$ there exists an integer $a_{r}$ with $(a_r,6Q\ell)=1$ such that
\[
	a_{r}^{2}(24t-1) \equiv 24(t + Q\ell r) - 1 \pmod{Q\ell^{j}}.
\]
Therefore $t+Q\ell r=t_A$ for some $A \in \Gamma_0(2Q\ell^j)$ as in \eqref{def:t-A}, so by Lemma~\ref{lem:M-t-t-A} we have $M_{Q\ell^j,t+Q\ell r} \equiv 0 \pmod{\ell}$. Since
\[
	\bigcup_{r=0}^{\ell^{j-1}-1} \{Q\ell^{j}n+ (t+Q\ell r)\} = \bigcup_{r=0}^{\ell^{j-1}-1} \{Q\ell(\ell^{j-1}n + r) + t\} = \{Q\ell n+t\},
\]
we conclude that $M_{Q\ell,t} \equiv 0 \pmod{\ell}$.
\end{proof}

We will use the following proposition to prove that $\legendre{24t-1}{\ell} \neq \legendre{-1}{\ell}$.

\begin{prop} \label{prop:m-q-l-exp}
	If $(Q,6\ell)=1$ and $t\equiv (1-Q^{2})/24 \pmod{\ell}$ then
	\begin{equation}
		M_{Q\ell,t}^{48Q\ell^2} \Delta^{s} \big|_{24Q\ell^2+12s} \pMatrix{1}{0}{2\ell}{1} = Q^{-24Q\ell^2}q^{-2Q^{2}\ell+s} + \cdots.
	\end{equation}
\end{prop}

Assume for the moment that the proposition is true and suppose, by way of contradiction, that $\legendre{24t-1}{\ell} = \legendre{-1}{\ell}$. We will construct an integer $t'$ satisfying both $t'=t_A$ for some $A \in \Gamma_0(2Q\ell)$ and $t' \equiv (1-Q^2)/24 \pmod{\ell}$. Let $\alpha$ be an integer satisfying
\[
	(24t-1)\alpha^{2} \equiv -1 \pmod{\ell}.
\]
Since $\ell \nmid \alpha Q$, there exists an integer $a$ with $(a,6Q\ell)=1$ such that
\begin{equation*} 
	a \equiv Q\alpha \pmod{\ell}. 
\end{equation*}
Let $t'$ be any integer satisfying
\begin{equation*} \label{eq:def-t'}
	a^{2}(24t-1) \equiv (24t'-1) \pmod{Q\ell}.
\end{equation*}
Then $t'=t_A$ for some $A \in \Gamma_0(2Q\ell)$ as in \eqref{def:t-A} and 
\begin{equation}
	t' \equiv \frac{1-Q^{2}}{24} \pmod{\ell}.
\end{equation}
By Proposition~\ref{prop:M-Q-ell} and Lemma~\ref{lem:M-t-t-A} we have
\begin{equation} \label{eqn:(qln+t')=0-mod-l}
	M_{Q\ell,t'} \equiv 0 \pmod{\ell}.
\end{equation}
For $s$ sufficiently large, define 
\[
	g':= M_{Q\ell,t'}^{48Q\ell^2} \Delta^{s} \in M_{24Q\ell^2+12s}(\Gamma_1(2Q\ell)).
\]
We have $g'\equiv 0 \pmod{\ell}$, so by Theorem~\ref{thm:deligne-mod-p} we have
\[
	g' \big|_{24Q\ell^2 +12s} \pMatrix{1}{0}{2\ell}{1} \equiv 0 \pmod{\ell},
\]
but this contradicts Proposition~\ref{prop:m-q-l-exp}.

\begin{proof}[Proof of Proposition~\ref{prop:m-q-l-exp}]

We compute the leading coefficient of $M_{Q\ell,t}$ at the cusp $1/2\ell$. By \eqref{eq:def-M-m-t} we have

\begin{align*}
	M_{Q\ell,t} \left( \pMatrix{1}{0}{2\ell}{1} z \right) 
	&= \frac{1}{Q\ell} \sum_{\lambda=0}^{Q\ell-1} \zeta_{Q\ell}^{-\lambda(t-1/24)} M \left( \pMatrix{1}{\lambda}{0}{Q\ell} \pMatrix{1}{0}{2\ell}{1} z \right) \\
	&= \frac{1}{Q\ell} \sum_{\lambda=0}^{Q\ell-1} \zeta_{Q\ell}^{-\lambda(t-1/24)} M \left( A_{\lambda} \pMatrix{d_{\lambda}}{\lambda'}{0}{Q\ell/d_{\lambda}} z \right)
\end{align*}
where $d_{\lambda} := (1+2\ell \lambda, Q)$, the integer $\lambda'$ satisfies the congruence
\[
	\frac{1+2\ell\lambda}{d_{\lambda}} \lambda' \equiv \lambda \pmod{Q\ell/d_{\lambda}},
\]
and $A_{\lambda}$ is the matrix
\[
	A_{\lambda} := \pMatrix{\dfrac{1+2\ell\lambda}{d_{\lambda}}} {\dfrac{-\frac{1+2\ell\lambda}{d_{\lambda}}\lambda'+\lambda}{Q\ell/d_{\lambda}}} {2\ell^{2}Q/d_{\lambda}} {-2\ell\lambda'+d_{\lambda}} \in \Gamma_{0}(2).
\]
Recall the transformation law for $M(z)$ in \eqref{eq:M-transform}. Since $M(z) = q^{-1/24} + \ldots$, the leading term of
\[
	(2\ell z + 1)^{-1/2} M_{Q\ell,t}\left( \pMatrix{1}{0}{2\ell}{1} z \right)
\]
arises from those $\lambda$ for which $d_{\lambda} = Q$. For these $\lambda$ we can take $\lambda'=Q\lambda$. The leading coefficient is
\[
	K = \frac{1}{Q\ell} \sum_{d_{\lambda} = Q} \zeta_{Q\ell}^{-\lambda(t-1/24)} \omega(A_{\lambda}) Q^{1/2} e^{-2\pi i Q\lambda/24\ell}.
\]
We must show that $K^{48Q\ell^2} = Q^{-24Q\ell^2}$, so it is enough to compute $K$ up to a $48Q\ell^2$-th root of unity. To this end, we will factor out any terms in the exponent of $\omega(A_\lambda)$ which are independent of $\lambda$ and collect them in roots of unity denoted $\omega_1, \omega_2,$ etc. If $A_{\lambda} = \pmatrix{a}{b}{c}{d}$, note that
\begin{align*}
	\frac{c+1+ad}{2} &= \ell^2 + 1 - 2\ell^2 \lambda^2 \equiv 0 \pmod{2}
\end{align*}
and
\begin{align*}
	-\frac{a+d}{24c} - \frac{a}{4} + \frac{3dc}{8} - \frac{Q\lambda}{24\ell}
	&= \frac{Q(36\ell^{4}-1)}{48\ell^{2}} - \frac{(1+2\ell\lambda)}{Q}\frac{(12\ell^{2}+1)}{48\ell^{2}} - \frac{3}{2}\ell^{3}Q\lambda \\
	&= - \lambda\frac{(12\ell^{2}+1)}{24Q\ell} - \frac{\lambda}{2} + \underbrace{\frac{3Q^2\ell^2-1}{4Q} - \frac{Q^2-1}{48Q\ell^2}}_{\text{independent of $\lambda$}} + \underbrace{\lambda\pfrac{1-3Q\ell^3}{2}}_{\text{an integer}}.
\end{align*}
Choose $\alpha \in \Z$ so that $\alpha Q \equiv 1 \pmod{2\ell}$ and $3\nmid\alpha$. Then, applying Lemma~\ref{lem:dedekind-sum} to the matrix $\pmatrix{\alpha}{*}{2\ell}{Q} \in \Gamma_0(\ell)$, we have
\[
	s(-Q + 2\ell\lambda', 2\ell^{2}) = s(-Q,2\ell^{2}) + \lambda' \frac{1-\alpha^{2}}{12\ell} + \text{an even integer}.
\]
Since $\lambda' \alpha \equiv \lambda \pmod{\ell}$ and  $1-Q^{2}\alpha^{2}\equiv 0\pmod{24\ell}$, we have
\begin{align*}
	\lambda'\frac{1-\alpha^{2}}{24\ell} &= \lambda'\frac{1-\alpha^{2}Q^{2}}{24\ell} - \frac{\lambda'\alpha^{2}}{\ell}\pfrac{1-Q^{2}}{24} \\
	&= - \frac{\lambda \alpha}{\ell} \left(\frac{1-Q^{2}}{24}\right) + \text{ an integer }.
\end{align*}
Since $6c\, s(-d,c)\in \Z$ (see \cite[Theorem 3.8]{Apostol}), the term $e^{\pi i s(-Q,2\ell^2)}$ is a $24\ell^2$-th root of unity. Therefore
\[
	K = \frac{\omega_1}{Q^{1/2}\ell} \sum_{d_{\lambda} = Q} \exp\left( -2\pi i \left( \frac{\lambda(t-1/24)}{Q\ell} - \frac{\lambda \alpha}{\ell}\pfrac{1-Q^{2}}{24} + \frac{\lambda}{2} +  \lambda\pfrac{12\ell^{2}+1}{24Q\ell} \right) \right).
\]
Let $\lambda_{0}$ be the smallest $\lambda$ for which $d_{\lambda}=Q$. Then each $\lambda$ is of the form $\lambda = \lambda_{0} + \delta Q$ for some $1\leq \delta \leq \ell-1$. Factoring out those terms which do not depend on $\delta$ we obtain
\begin{equation*}
	K = \frac{\omega_{2}}{Q^{1/2}\ell} \sum_{\delta=0}^{\ell-1} \exp\left( -2\pi i \, \delta \left( \frac{t}{\ell} - \frac{Q\alpha}{\ell} \left(\frac{1-Q^{2}}{24}\right) + \frac{Q+\ell}{2} \right) \right).
\end{equation*}
Note that $(Q+\ell)/2$ is an integer. Since $Q\alpha \equiv 1 \pmod{\ell}$ and $(1-Q^2)/24$ is an integer, we obtain
\begin{align*}
	K &= \frac{\omega_{2}}{Q^{1/2}\ell} \sum_{\delta=0}^{\ell-1} \exp\left( -2\pi i \frac{\delta}{\ell} \left( t- \frac{1-Q^{2}}{24} \right) \right).
\end{align*}
By assumption, $t\equiv (1-Q^{2})/24 \pmod{\ell}$ so each term in the sum is $1$. Therefore
\begin{align*}
	K^{48Q\ell^2} &= Q^{-24Q\ell^2},
\end{align*}
as desired.
\end{proof}

\section{Proof of Theorem~\ref{thm:omega-q}}

Define
\begin{gather} \label{eq:def-O-m-t}
	\Omega_{m,t}(z) :=  \frac{1}{m} \sum_{\lambda=0}^{m-1} \zeta_{m}^{-\lambda(t+2/3)} \Omega \pfrac{z+\lambda}{m}.
\end{gather}
In this case we call a progression $t\pmod{m}$ \emph{good} if for some $p|m$ we have $\legendre{-3t-2}{p} = -1$. By \cite{AK}, if $t\pmod{m}$ is good, $\Omega_{m,t}$ is weakly holomorphic. If $t \pmod m$ is not good, we can reduce to the case of a good progression $T \pmod{mp}$ as before, so to prove Theorem~\ref{thm:omega-q} we may assume that $t \pmod{m}$ is a good progression. A calculation shows that if $t\pmod{m}$ is good then
\[
	\Omega_{m,t} = q^{\frac{t+2/3}{m}} \sum c(mn+t) q^{n}.
\]
Thus, as in the $f(q)$ case, the second statement in Theorem~\ref{thm:omega-q} follows from the first statement.

We will need results from \cite{AK} which describe the transformation of $\Omega_{m,t}$ under the action of $\Gamma_{0}(2N_{m})$. If $A = \pmatrix{a}{b}{c}{d} \in \Gamma_{0}(2N_{m})$ has $3\nmid a$ then we define $t_{A}$ to be any integer satisfying
\begin{equation} \label{def:O-t-A}
	t_{A} \equiv ta^{2} + \frac{2}{3}(a^{2}-1) \pmod{m}.
\end{equation}

\begin{prop} \label{prop:O-m-t-O-m-t-A}
	For every $A = \pmatrix{a}{b}{c}{d} \in \Gamma_{0}(2N_{m})$ with $3\nmid a$ we have
	\[
		\Omega_{m,t}^{24m} \big|_{12m} A = \Omega_{m,t_{A}}^{24m}.
	\]
\end{prop}

\begin{prop} \label{prop:O-m-t-modular}
	Suppose that $t \pmod m$ is good. Then there exists $s \in \N$ such that
	\[
		\Delta^{s} \, \Omega_{m,t}^{24m} \in M_{12(m+s)}(\Gamma_{1}(2N_{m})).
	\]
\end{prop}

We will need an analogue of Lemma~\ref{lem:M-t-t-A} which can be proved similarly.

\begin{lem} \label{lem:O-t-t-A}
	Suppose that $t \pmod{m}$ is good. Let $A=\pmatrix{a}{b}{c}{d} \in \Gamma_{0}(2N_{m})$ with $3\nmid a$, and let $t_A$ be as in \eqref{def:O-t-A}. If
		$\Omega_{m,t} \equiv 0 \pmod{\ell}$
	then $\Omega_{m,t_A} \equiv 0 \pmod{\ell}$.
\end{lem}

We now show that $\ell | m$ and reduce to the case $m=Q\ell$ with $(Q,3\ell)=1$.

\begin{prop}
	Suppose that $t\pmod{m}$ is good and write $m=3^v \ell^j Q$ with $(Q,3\ell)=1$. If $\Omega_{m,t} \equiv 0 \pmod{\ell}$ then $j > 0$. If additionally $\ell \nmid (3t+2)$ then $\Omega_{Q\ell,t}\equiv 0 \pmod{\ell}$.
\end{prop}

\begin{proof}
Applying Lemma~\ref{lem:radu} with $B=16$ and $N=1$ we see that as $a$ ranges over integers with $(a,6m)=1$ the quantity $t_{A}$ covers each of the progressions $t+\lambda Q\ell^j \pmod m$, for $0\leq \lambda < 3^{v}$. So, as in the proof of Proposition~\ref{prop:M-Q-ell}, we can conclude that $\Omega_{Q\ell^j,t} \equiv 0 \pmod{\ell}$.

To show that $j>0$, suppose by way of contradiction that $j=0$, and define
\[
	h := \Delta^{s} \, \Omega_{Q,t}^{24Q} \in M_{12(Q+s)}(\Gamma_{1}(2Q)).
\]
Then by \cite[Proposition 12]{AK}, we have
\[
	h \big|_{12(Q+s)} \pMatrix{1}{0}{1}{1} = (-1)^{Q} (2Q)^{-12Q} q^{-Q^{2}/2+s} + \cdots.
\]
By assumption, $\Omega_{Q,t} \equiv 0 \pmod{\ell}$, so the modular form $\ell^{-24Q} h$ has integral coefficients. By Theorem~\ref{thm:deligne-coeff}, the coefficients of 
\[
	\ell^{-24Q} h \big|_{12(Q+s)} \pMatrix{1}{0}{1}{1}
\]
lie in $\Z[1/4Q,\zeta_{4Q}]$, which contradicts $\ell \nmid Q$.

We show that $\Omega_{Q\ell,t} \equiv 0 \pmod{\ell}$. By the argument in \cite[Lemma 4.6]{Radu:ao}, for each $r$ with $0\leq r < \ell^{j-1}$ there exists an integer $a_{r}$ with $(a_r,6Q\ell)=1$ such that
\[
	a_{r}^{2}(3t+2) \equiv 3(t + Q\ell r) + 2 \pmod{Q\ell^{j}}.
\]
So $t+Q\ell r = t_A$ for some $A \in \Gamma_0(16Q\ell^j)$ as in \eqref{def:O-t-A}. Therefore, as in the proof of Proposition~\ref{prop:M-Q-ell} we can conclude that $\Omega_{Q\ell,t} \equiv 0 \pmod{\ell}$.
\end{proof}

We will use the following proposition to prove that $\legendre{3t+2}{\ell} \neq \legendre{-1}{\ell}$.

\begin{prop} \label{prop:O-q-l-exp}
	If $4|Q$, $(Q,3\ell)=1$ and $t\equiv -(Q^{2}+32)/48 \pmod{\ell}$ then
	\begin{equation}
		\Omega_{Q\ell,t}^{48Q\ell^2} \Delta^{s} \big|_{24Q\ell^2+12s} \pMatrix{1}{0}{\ell}{1} = (2Q)^{-24Q\ell^2} q^{-Q^2 \ell + s} + \cdots.
	\end{equation}
\end{prop}

Assume for the moment that the proposition is true and assume, without loss of generality, that $4|Q$. Suppose, by way of contradiction, that $\legendre{3t+2}{\ell} = \legendre{-1}{\ell}$. We will construct an integer $t'$ that satisfies both $t=t_A$ for some $A\in \Gamma_0(16Q\ell)$ and $t \equiv -(Q^2+32)/48 \pmod{\ell}$. If $\legendre{3t+2}{\ell} = \legendre{-1}{\ell}$, then there exists an $\alpha \in \Z$ with 
\[
	(3t+2)\alpha^{2} \equiv -1 \pmod{\ell}.
\]
Since $\ell \nmid \alpha$ and $(Q,3\ell)=1$, there exists $a\in \Z$ with $(a,3Q\ell)=1$ such that
\begin{equation*} 
	4a \equiv Q\alpha \pmod{\ell}. 
\end{equation*}
Let $t'$ be any integer satisfying
\begin{equation*} \label{eq:w-def-t'}
	a^{2}(3t+2) \equiv (3t'+2) \pmod{Q\ell}.
\end{equation*}
Then $t'=t_A$ for some $A \in \Gamma_0(16Q\ell)$ as in \eqref{def:O-t-A} and
\begin{equation}
	t' \equiv -\frac{(Q/4)^{2}+2}{3} \equiv  -\frac{Q^{2}+32}{48} \pmod{\ell}.
\end{equation}
So by Lemma~\ref{lem:O-t-t-A} we have $\Omega_{Q\ell,t'} \equiv 0 \pmod{\ell}$. Define 
\[
	h':= \Omega_{Q\ell,t'}^{48Q\ell^2} \Delta^{s} \in M_{24Q\ell^2+12s}(\Gamma_{1}(4Q\ell)),
\]
We have $h' \equiv 0 \pmod{\ell}$, so by Theorem~\ref{thm:deligne-mod-p} we have
\[
	h' \big|_{24Q\ell^2+12s} \pMatrix{1}{0}{\ell}{1} \equiv 0 \pmod{\ell},
\]
but this contradicts Proposition~\ref{prop:O-q-l-exp}.

\begin{proof}[Proof of Proposition~\ref{prop:O-q-l-exp}]
	By \eqref{eq:def-O-m-t} we have
\begin{align*}
	\Omega_{Q\ell,t} \left( \pMatrix{1}{0}{\ell}{1} z \right) 
	&= \frac{1}{Q\ell} \sum_{\lambda=0}^{Q\ell-1} \zeta_{Q\ell}^{-\lambda(t+2/3)} \Omega\left( \pMatrix{1}{\lambda}{0}{Q\ell} \pMatrix{1}{0}{\ell}{1} z \right) \\
	&= \frac{1}{Q\ell} \sum_{\lambda=0}^{Q\ell-1} \zeta_{Q\ell}^{-\lambda(t+2/3)} \Omega\left( B_{\lambda} \pMatrix{d_{\lambda}}{\lambda'}{0}{Q\ell / d_{\lambda}} z \right),
\end{align*}
where $d_{\lambda} := (1+\ell\lambda,Q)$, the integer $\lambda'$ is chosen to satisfy $\frac{1+\ell\lambda}{d_{\lambda}} \lambda' \equiv \lambda \pmod {Q\ell / d_{\lambda}}$, and
\[
	B_{\lambda} := 
	\pMatrix
	{\dfrac{1+\ell\lambda}{d_{\lambda}}}
	{\dfrac{-\frac{1+\ell \lambda}{d_{\lambda}} \lambda' + \lambda}{Q\ell / d_{\lambda}}}
	{Q\ell^{2} / d_{\lambda}}
	{d_{\lambda} - \ell \lambda'} \in \SL_{2}(\Z).
\]
Recall that the transformation law \eqref{eq:O-transform} for $\Omega(z)$ depends on the parity of the lower entries of $B_\lambda$. Note that if $Q/d_\lambda$ is odd, then $d_\lambda$ is even since $4|Q$. In these cases, we choose $\lambda'$ to be even so that $d_\lambda - \ell \lambda'$ is also even. By \eqref{eq:O-transform} we have
\begin{align*}
	\Omega_{Q\ell,t}\left( \pMatrix{1}{0}{\ell}{1} z \right) &= 
	\frac{1}{Q\ell} \sum_{Q/d_\lambda \text{ even }} \zeta_{Q\ell}^{-\lambda(t+2/3)} \omega_1(B_\lambda) \, d_\lambda^{1/2} \, (\ell z+1)^{1/2} \, \Omega\pfrac{d_\lambda z+\lambda'}{Q\ell/d_\lambda} \\
	&+ \frac{1}{Q\ell} \sum_{Q/d_\lambda \text{ odd }} \zeta_{Q\ell}^{-\lambda(t+2/3)} \pfrac{d_\lambda}{2}^{1/2} \omega_2(B_\lambda) \, (\ell z+1)^{1/2} M\pfrac{d_\lambda z+\lambda'}{2Q\ell/d_\lambda}.
\end{align*}
Since $\Omega(z) = 2q^{2/3}+\ldots$ and $M(z/2) = q^{-1/48}+\ldots$, the leading term of
\[
	(\ell z + 1)^{-1/2} \Omega_{Q\ell,t}\left( \pMatrix{1}{0}{\ell}{1} z \right)
\]
arises from those $\lambda$ for which $d_{\lambda} = Q$. For these $\lambda$ we may take $\lambda' = Q\lambda$ so that
\begin{equation} \label{eq:B-lambda}
	B_\lambda = \pMatrix{\frac{1+\ell\lambda}{Q}}{-\lambda^2}{\ell^2}{Q(1-\ell\lambda)}.
\end{equation}
Therefore
\[
	(\ell z + 1)^{-1/2} \Omega_{Q\ell,t}\left( \pMatrix{1}{0}{\ell}{1} z \right) = K q^{-Q/48\ell} + \cdots,
\]
where
\[
	K = \frac{1}{Q\ell} \sum_{d_{\lambda} = Q} \zeta_{Q\ell}^{-\lambda(t+2/3)} \pfrac{Q}{2}^{1/2} \omega_{2}(B_{\lambda}) \, e^{-\frac{2\pi i \lambda Q}{48\ell}}.
\]
We will compute $K$ up to $48Q\ell^2$-th roots of unity denoted $\omega_1, \omega_2,$ etc. Choose $\alpha \in \Z$ such that $\alpha Q \equiv 2 \pmod{\ell}$ and $3 \nmid \alpha$. Applying Lemma~\ref{lem:dedekind-sum} to the matrix $\pmatrix{\alpha}{*}{\ell}{Q/2}$, we obtain
\[
	s(Q\ell \lambda/2 - Q/2, \ell^{2}) = s(-Q/2, \ell^{2}) + \frac{Q\lambda}{2} \cdot \frac{1-\alpha^{2}}{12\ell} + \text{an even integer}.
\]
Since $6\ell^2 \, s(-Q/2,\ell^2) \in \Z$, the term $e^{-\pi i s(-Q/2,\ell^2)}$ is a $12\ell^2$-th root of unity which is independent of $\lambda$. If $B_{\lambda} = \pmatrix{a}{b}{c}{d}$ then by \eqref{eq:B-lambda} we obtain
\begin{align*}
	\frac{32a-d}{48c} - &\frac{1}{4}(2a+b-3-3ab+3a/c) \\
	&= \frac{\lambda(Q^{2}+32)}{48Q\ell} - \frac{3\lambda}{4Q\ell} + \frac{\lambda^{2}}{4} - \frac{\lambda(3\lambda+2\ell+3\lambda^{2}\ell)}{4Q} + \underbrace{\frac{1}{2Q} + \frac{3}{4} - \frac{Q^2+4}{48Q\ell^2}}_{\text{independent of $\lambda$}}.
\end{align*}
So we have
\begin{align*}
	K = \frac{\omega_1}{\ell\sqrt{2Q}} \sum_{d_{\lambda} = Q} e^{2\pi i R_\lambda},
\end{align*}
where
\begin{align*}
	R_\lambda &=-\frac{\lambda(t+2/3)}{Q\ell} + \frac{\lambda(Q^{2}+32)}{48Q\ell} - \frac{Q\lambda(1-\alpha^{2})}{48\ell} - \frac{3\lambda}{4Q\ell} + \frac{\lambda^{2}}{4} - \frac{\lambda(3\lambda+2\ell+3\lambda^{2}\ell)}{4Q} - \frac{Q\lambda}{48\ell} \\
	&= -\frac{t\lambda}{Q\ell} - \frac{\lambda(Q^{2}+32)}{48Q\ell} + \frac{\lambda\left( \alpha^{2}Q^{2} - 4 + 12\ell(Q\lambda - 3\lambda - 2\ell - 3\ell\lambda^{2}) \right)}{48Q\ell}.
\end{align*}
The $\lambda$ for which $d_{\lambda} = Q$ are of the form $\lambda = \lambda_{0} + \delta Q$ for $0\leq \delta <\ell$. Replacing $\lambda$ by $\lambda_{0} + \delta Q$ and recalling that $4|Q$, then ignoring the terms which are integers, we obtain
\begin{equation*}
	K = \frac{\omega_2}{\ell\sqrt{2Q}} \sum_{\delta=0}^{\ell-1} e^{2\pi i R'_\delta},
\end{equation*}
where
\begin{equation} \label{O-comp-exp}
	R'_\delta = -\frac{t\delta}{\ell} - \frac{\delta(Q^{2}+32)}{48\ell} + \frac{\delta}{\ell} \cdot  \frac{\alpha^{2}Q^{2} - 4 - 12\ell(9\ell\lambda_{0}^{2} + 6\lambda_{0} + 2\ell)}{48}.
\end{equation}
Since $d_{\lambda_0}=(1+\ell\lambda_0,Q)=Q$ and $4|Q$, we have $1+\ell\lambda_{0} \equiv 0 \pmod{4}$. Thus
\[
	9\ell\lambda_{0}^{2} + 6\lambda_{0} + 2\ell \equiv \ell \pmod{4},
\]
which implies that
\begin{equation} \label{Rp-delta}
	R'_\delta = -\frac{t\delta}{\ell} - \frac{\delta(Q^{2}+32)}{48\ell} + \frac{\delta}{\ell} \cdot \frac{\alpha^{2}Q^{2} - 4 - 12\ell^{2}}{48} + \text{ an integer}.
\end{equation}
Recall that we chose $\alpha$ so that $\alpha^{2}\equiv 1\pmod{3}$ and $\alpha Q \equiv 2 \pmod{\ell}$. These facts, together with $4|Q$, imply that
\[
	\alpha^{2}Q^{2} - 4 - 12\ell^{2} \equiv 0 \pmod{48\ell},
\]
so the third term in \eqref{Rp-delta} is also an integer. We assumed that $t\equiv -(Q^{2}+32)/48 \pmod{\ell}$, so in fact $R'_\delta \in \Z$. Therefore
\[
	K = \frac{\omega_2}{\ell\sqrt{2Q}} \sum_{\delta=0}^{\ell-1} 1 = \frac{\omega_{2}}{\sqrt{2Q}}. \qedhere
\]
\end{proof}

\section{Weakly holomorphic modular forms} \label{sec:weak-hol}

In this section we state and prove a general theorem about a large class of weakly holomorphic modular forms which contains $\eta$-quotients. For an integer or half-integer $k$, integers $B$ and $N$ (with $4|N$ if $k \notin \Z$), and a Dirichlet character $\chi$ modulo $N$, we define
\[
	\mathcal{S}(B,k,N,\chi) 
		:= \left\{ \eta^{B}(z) F(z) : F(z) \in M_{k}^{!}(\Gamma_{0}(N),\chi) \right\}.
\]
Here $M_k^!(\Gamma_0(N),\chi)$ denotes the space of meromorphic modular forms of weight $k$ with Nebentypus $\chi$ on $\Gamma_0(N)$ whose poles, if any, are supported at the cusps. A form $F\in M_k^!(\Gamma_0(N),\chi)$ transforms under $\gamma=\pmatrix abcd \in \Gamma_0(N)$ as
\[
	F \big|_k \gamma = \rho(\gamma) \chi(d) F,
\]
where
\[
	\rho(\gamma) := 
	\begin{cases}
		1 & \text{ if } k\in\Z, \\
		\epsilon_d^{-2k} \legendre cd ^{2k} &\text{ if } k \in \frac{1}{2} + \Z,
	\end{cases}
\]
and
\[
	\epsilon_d := 
	\begin{cases}
		1 & \text{ if } d\equiv 1\pmod{4}, \\
		i & \text{ if } d\equiv 3\pmod{4}
	\end{cases}
\]
(see \cite[Chapter 1]{Ono:web} for details). The following theorem generalizes \cite[Theorem 1.2]{Radu:ao} for any $f\in \mathcal S(B,k,N,\chi)$.

\begin{thm} \label{thm:weak-hol}
	Let $\ell \geq 5$ be prime. Suppose that 
\[
	f(z) = q^{B/24} \sum_{n\geq n_{0}} a_{f}(n) q^{n} \in \mathcal{S}(B,k,N,\chi)
\]
has a pole at $\infty$, leading coefficient equal to $1$, and rational $\ell$-integral coefficients. Suppose that $\ell \nmid N(24n_{0}+B)$ and $(m_B,N)=1$ \textup{(}with $m_B$ defined in \eqref{eq:def-m-B}\textup{)}. If
\[
	a_{f}(mn+t) \equiv 0 \pmod{\ell},
\]
then $\ell | m$ and $\legendre{24t+B}{\ell} \neq \legendre{24n_{0}+B}{\ell}$.
\end{thm}

Before proving Theorem~\ref{thm:weak-hol}, we show that Theorem~\ref{thm:eta} is a corollary.

\begin{proof}[Proof of Theorem~\ref{thm:eta}]
Suppose that $f(z)$ is the $\eta$-quotient
\[
	f(z) = \prod_{\delta|N} \eta(\delta z)^{r_{\delta}} = q^{B/24} \sum_{n=0}^\infty a_f(n)q^n,
\]
and that $a_f(mn+t)\equiv 0 \pmod{\ell}$. Recall that
\begin{equation} \label{eq:def-B}
	B = \sum_{\delta|N} \delta r_{\delta},
\end{equation}
and write $f(z) = \eta^B(z) F(z)$ with 
\[
	F(z) = \frac{f(z)}{\eta^B(z)} 
	= \eta^{-B}(z) \prod_{\delta | N} \eta(\delta z)^{r_\delta}.
\]
To apply Thoerem \ref{thm:weak-hol} we use a standard criterion \cite{Newman:eta} to show that $F(z) \in M_k^!(\Gamma_0(N^4),\chi)$ for some character $\chi$. In light of \eqref{eq:def-B}, the condition
\begin{equation} \label{cond1}
	\sum_{\delta|N} \delta r_{\delta} - B \equiv 0 \pmod{24}
\end{equation}
is satisfied trivially. The condition
\begin{equation} \label{cond2}
	N^4 \left(\sum_{\delta|N} \frac{r_{\delta}}{\delta} - B\right) \equiv 0 \pmod{24}
\end{equation}
is also satisfied (to see this, consider cases depending on the value of $(N,6)$). We apply Theorem~\ref{thm:weak-hol} with $n_0=0$ to obtain $\ell | m$ and $\legendre{24t+B}{\ell} \neq \legendre{B}{\ell}$.
\end{proof}

\subsection*{Proof of Theorem~\ref{thm:weak-hol}}

Suppose that 
\[
	f(z) = q^{B/24} \sum_{n=n_{0}}^{\infty} a_{f}(n) q^{n} = q^{n_{0}+B/24} + \cdots \in \S(B,k,N,\chi)
\]
has a pole at $\infty$ and rational $\ell$-integral coefficients, and that $\ell \nmid N(24n_0+B)$. Given $m$ and $t$ with $(m_B,N)=1$, define
\begin{align} \label{eq:def-f-m-t}
	f_{m,t} &:= \frac{1}{m} \sum_{\lambda=0}^{m-1} \zeta_{m}^{-\lambda(t+B/24)} f\left(\frac{z+\lambda}{m} \right) \\
	&\hphantom{:}= q^{\frac{t+B/24}{m}} \sum a_{f}(mn+t) q^{n}. \nonumber
\end{align}

We require a transformation law for $f_{m,t}$. Define
\[
	N_{m} = 
	\begin{cases}
		m & \text{if } (m,6) = 1, \\
		8m & \text{if } (m,6) = 2, \\
		3m & \text{if } (m,6) = 3, \\
		24m & \text{if } (m,6) = 6.
	\end{cases}
\]
Note that this differs slightly from the previous definition of $N_{m}$. If $A=\pmatrix{a}{b}{c}{d} \in \Gamma_{0}(N_{m})$ has $(a,6)=1$ then we define $t_{A}$ to be any integer satisfying
\begin{equation} \label{def:wh-t-A}
	t_{A} \equiv ta^{2} + B \frac{a^{2}-1}{24} \pmod{m}.
\end{equation}

The following is proved in \cite[Section 5]{AK}.

\begin{prop} \label{lem:transform-f-m-t}
Suppose $f$ is as above, and define $\kappa = 24mN(k+B/2)$. For every $A = \pmatrix{a}{b}{c}{d} \in \Gamma_{0}(N N_{m})$ with $(a,6)=1$ we have
\[
	f_{m,t}^{24mN} \big|_{\kappa} A = f_{m,t_{A}}^{24mN}.
\]
In particular, there exists $s\in \N$ such that
\[
	\Delta^{s} f_{m,t}^{24mN} \in M_{\kappa + 12s} (\Gamma_{1}(N N_{m})).
\]
\end{prop}

We need the following analogue of Lemma~\ref{lem:M-t-t-A}. The proof is similar so we omit it here.

\begin{lem} \label{thm:t-t'} \label{lem:wh-t-t-A}
Suppose that $f$ is as above and that $A = \pmatrix{a}{b}{c}{d} \in \Gamma_{0}(NN_{m})$ has $(a,6)=1$, and define $t_A$ as in \eqref{def:wh-t-A}. If $f_{m,t}\equiv 0\pmod{\ell}$ then $f_{m,t_A} \equiv 0\pmod{\ell}$.
\end{lem}

For the remainder of this section, assume that $f_{m,t}\equiv 0 \pmod{\ell}$. Proposition \ref{prop:f-Q-ell} will show that $\ell | m$ and reduce us to a simpler case. We fix some notation. As in Section \ref{sec:preliminaries}, write $B=2^r3^sB'$ and $m=2^u3^vm'$ with $(B',6)=(m',6)=1$, and write $Q_{m,B}=2^\alpha3^\beta m'=Q\ell^j$ with $(Q,\ell)=1$. Recall that $\alpha$ and $\beta$ are defined in \eqref{eq:def-alpha-beta} by
\begin{equation} \label{eq:def-alpha-beta-2}
	\alpha := 
	\begin{cases}
		0 &\text{ if } r=0,\\
		\min(r,u) &\text{ if } r=1,2,\\
		u &\text{ if } r\geq 3,
	\end{cases}
	\qquad \qquad
	\beta :=
	\begin{cases}
		0 &\text{ if }s=0,\\
		v &\text{ if }s\geq 1.
	\end{cases}
\end{equation}
Since $\ell \nmid N$ and the primes dividing $Q_{m,B}$ are the same as those dividing $m_B$, we have $(Q,N)=1$. There are two main cases to consider depending on the value of $r$. To facilitate this, define parameters $M \in \N$ and $\varepsilon \in \{0,1,2\}$ as follows. If $r=1,2$, define $M:=\frac{24N}{(B,3)}$ and $\varepsilon:=\alpha=\min(r,u)$. If $r\neq 1,2$, define $M:=\frac{24N}{(B,24)}$ and $\varepsilon:=0$.

\begin{prop}\label{prop:f-Q-ell}
Assume the notation above. If $f_{m,t}\equiv 0 \pmod{\ell}$, then $j > 0$. If additionally $\ell \nmid (24t+B)$ then $f_{Q\ell,t}\equiv 0 \pmod{\ell}$.
\end{prop}

\begin{proof}
Applying Lemma~\ref{lem:radu} we see that for each $\lambda$ with $0\leq \lambda < m/Q\ell^j$ there exists $a \in \Z$ with $(a,6mN)=1$ such that
\[
	t + \lambda Q\ell^j \equiv t a^{2} + B\frac{a^{2}-1}{24} \pmod{m}.
\]
As in the proof of Proposition~\ref{prop:M-Q-ell}, we conclude by Lemma~\ref{lem:wh-t-t-A} that $f_{Q\ell^j,t} \equiv 0 \pmod{\ell}$.

Suppose now that $j = 0$. From \cite[Section 5]{AK}, the leading coefficient in the expansion of $f_{Q,t}$ at the cusp $1/N$ is 
\[
	\xi Q^{B/2+k-1},
\]
where $\xi$ is a $24QN$th root of unity. Using Theorem~\ref{thm:deligne-coeff} as before with the modular form $\ell^{-24QN}\Delta^{s} f_{Q,t}^{24QN}$, which has integral coefficients, we obtain a contradiction. Therefore $j>0$.

Lastly, we use \cite[Lemma 4.6]{Radu:ao} as in the proof of Proposition~\ref{prop:M-Q-ell} to conclude that $f_{Q\ell,t} \equiv 0 \pmod{\ell}$ if $\ell \nmid (24t+B)$.
\end{proof}

We will use the following to prove that $\legendre{24t+B}{\ell} \neq \legendre{24n_{0}+B}{\ell}$.

\begin{prop} \label{prop:wh-exp-at-cusp}
Assume the notation above. Let $\kappa = 24Q\ell N(k+B/2)$ and define $R:=Q/2^\varepsilon$. If
\[
	t\equiv R^{2}n_{0}+ B\frac{R^{2}-1}{24} \pmod{2^{\varepsilon}\ell},
\] 
then
\[
	\Delta^{s} f_{Q\ell,t}^{24Q\ell N} \mid_{\kappa+12s} \pMatrix{1}{0}{M\ell}{1} 
	= R^{\kappa-24Q\ell N} q^{QRNB+24RNn_{0}+s} + \cdots.
\]
\end{prop}

Assume for the moment that the proposition is true. Suppose, by way of contradiction, that $\legendre{24t+B}{\ell} = \legendre{24n_{0}+B}{\ell}$. We will construct an integer $t'$ satisfying both $t'=t_A$ for some $A\in \Gamma_0(NN_{Q\ell})$ and $t'\equiv R^2n_0+B(R^2-1)/24 \pmod{2^\varepsilon\ell}$. There exists an $a' \in \Z$ with 
\[
	(24t+B)a'^{2} \equiv 24n_{0}+B \pmod{\ell},
\]
and since $2^{\varepsilon}|(B,24)$, we have
\[
	(24t+B)a'^{2} \equiv 24n_{0}+B \pmod{2^{\varepsilon}\ell}.
\]
Let $a$ be an integer with $(a,6Q\ell)=1$ such that $a \equiv Ra' \pmod{\ell}$ and let $t'$ be any integer satisfying
\begin{equation*} 
	a^{2}(24t+B) \equiv (24t'+B) \pmod{Q\ell}.
\end{equation*}
Then
\begin{equation}
	t' \equiv R^{2}n_{0} + B \frac{R^{2}-1}{24} \pmod{2^{\varepsilon}\ell}
\end{equation}
and by Lemma~\ref{lem:wh-t-t-A} we have $f_{Q\ell,t'} \equiv 0 \pmod{\ell}$, but this is not compatible with Proposition~\ref{prop:wh-exp-at-cusp}.

\begin{proof}[Proof of Proposition~\ref{prop:wh-exp-at-cusp}] 
By \eqref{eq:def-f-m-t} we have
\begin{equation} \label{eq:f-qell-t-Mell}
	f_{Q\ell,t}\left(\pMatrix{1}{0}{M\ell}{1} z\right) = \frac{1}{Q\ell}\sum_{\lambda=0}^{Q\ell-1} \zeta_{Q\ell}^{-\lambda(t+B/24)} f\left(C_{\lambda} \pMatrix{d_{\lambda}}{\lambda'}{0}{Q\ell/d_{\lambda}} z\right),
\end{equation}
where $d_{\lambda} = (1+ \lambda \ell M, Q)$, the integer $\lambda'$ satisfies
\[
	\frac{1+ \lambda \ell M}{d_{\lambda}} \lambda' \equiv \lambda \pmod{Q\ell/d_{\lambda}},
\]
and $C_{\lambda}$ is the matrix
\[
	\pMatrix{ \dfrac{1+\lambda \ell M}{d_{\lambda}} }{ \dfrac{-\frac{1+\lambda \ell M}{d_{\lambda}}\lambda' + \lambda}{Q\ell/d_{\lambda}} }{M Q\ell^{2}/d_{\lambda}}{d_{\lambda}-M\ell\lambda'} \in \Gamma_{0}(N).
\]
By assumption, $f$ has a pole at $\infty$, so the leading term of 
\[
	(M\ell z+1)^{-k-B/2} f_{Q\ell, t} \left(\pMatrix{1}{0}{M\ell}{1} z\right)
\]
arises from those $\lambda$ for which $d_\lambda$ is maximized. Since $(Q,N)=1$ and since $3|Q$ only if $3|B$, we see that $(Q,M)=2^\varepsilon$. Thus, the largest value of $d_\lambda$ is $R$. For these $\lambda$ we can take $\lambda' = R\lambda$, so that
\[
	C_\lambda = \pMatrix{\frac{1+\lambda \ell M}{R}}{-\frac{M\lambda^2}{2^\varepsilon}}{2^\varepsilon M\ell^2}{R(1-M\ell\lambda)}.
\]

We compute the terms of \eqref{eq:f-qell-t-Mell} for which $d_\lambda=R$. Recall that $f(z) = \eta^{B}(z)F(z)$. We consider first the transformation of $F(z)$. We have
\begin{equation} \label{eq:transform-F}
	(M\ell z + 1)^{-k} F\left(C_{\lambda}\pMatrix{R}{\lambda R}{0}{2^{\varepsilon}\ell}z\right) 
	= \rho(C_{\lambda}) \, R^{k} \, \chi(R)\chi(1-M\ell\lambda) \, F\left(\pMatrix{R}{\lambda R}{0}{2^{\varepsilon}\ell}z\right),
\end{equation}
where
\[
	\rho(C_{\lambda}) = 
	\begin{cases}
		1 & \text{ if $k\in\Z$,} \\
		\epsilon_{R-M\ell\lambda R}^{-2k}\legendre{2^{\varepsilon}M\ell^{2}}{M\ell\lambda R - R}^{2k} & \text{ if $k\in \frac{1}{2}+\Z$.}
		\end{cases}
\]
Note that since $N|M$, we have $\chi(1-M\ell\lambda) = 1$.
We show that $\rho(C_{\lambda})$ is independent of $\lambda$. Suppose that $k\in \frac{1}{2}+\Z$. Then $4|N$, so $R(1-M\ell\lambda) \equiv R \pmod{4}$, thus $\epsilon_{R-M\ell\lambda R} = \epsilon_{R}$. Write $M=2^{\mu}M'$ with $M'$ odd and $\mu\geq 2$. We have
\[
	\legendre{2^{\varepsilon}M\ell^{2}}{M\ell\lambda R - R} = \legendre{2^\varepsilon M}{R} \legendre{2}{M\ell\lambda-1}^{\varepsilon+\mu} \legendre{M'}{M\ell\lambda-1} = \legendre{2^\varepsilon M}{R},
\]
since $\varepsilon=0$ when $\mu=2$.
Therefore $\rho(C_\lambda)=\rho$ is independent of $\lambda$.

We now consider the transformation of $\eta^{B}(z)$. We have
\begin{equation} \label{eq:transform-eta-B}
	(M\ell z+1)^{-B/2} \, \eta^{B}\left(C_{\lambda}\pMatrix{R}{\lambda R}{0}{2^{\varepsilon}\ell} z\right)
		= \xi(C_{\lambda})^{B} \, R^{B/2} \, \eta^{B}\left(\pMatrix{R}{\lambda R}{0}{2^{\varepsilon}\ell} z\right),
\end{equation}
with $\xi(C_{\lambda})$ defined in \eqref{eq:eta-mult}. Write $C_{\lambda} = \pmatrix{a}{b}{c}{d}$. If $B$ is even then the factor $(\frac{c}{|d|})^B$ or $(\frac{d}{|c|})^B$ is equal to $1$. Suppose that $B$ is odd. Then $\varepsilon=0$, $8|M$ and $c$ is even. In this case we have
\[
	\legendre{c}{|d|} = \legendre{M\ell^2}{M\ell \lambda R-R} = \legendre{M}{R} \legendre{2}{M\ell\lambda-1}^\mu \legendre{M'}{M\ell\lambda-1} = \legendre{M}{R}
\] 
by a similar argument as before. Note that $24|Bc$ by the definition of $M$. If $2^\varepsilon M$ is even, we have
\begin{align*}
	\exp \left( \tfrac{2\pi i}{24}B [ (a+d)c - bd(c^2-1) + 3d - 3 - 3cd ] \right)
		&= \exp \left( \tfrac{2\pi i}{24} [Bbd + 3Bd - 3B] \right) \\
		&= \zeta_8^{B(R-1)} \exp \left( \tfrac{2\pi i}{24} Bbd \right)
\end{align*}
since $Bd \equiv BR \pmod{24}$. If $2^\varepsilon M$ is odd, we have
\begin{align*}
	\exp \left( \tfrac{2\pi i}{24} B [ (a+d)c - bd(c^2-1) - 3c ] \right) 
		= \exp \left( \tfrac{2\pi i}{24} Bbd \right).
\end{align*}
It remains to show that $24|Bbd$. We have
\[
	Bbd = -\frac{BM}{2^\varepsilon} \lambda^2 R (1-M\ell\lambda)  \equiv 0 \pmod{24}
\]
since $2^\varepsilon \cdot 24 | BM$.
So, in each case $\xi(C_\lambda)^B=\xi^B$ is independent of $\lambda$.

From \eqref{eq:f-qell-t-Mell}, \eqref{eq:transform-F}, and \eqref{eq:transform-eta-B} we obtain
\begin{align} \label{eq:sum1}
	f_{Q\ell,t} \big|_{k+B/2} \pMatrix{1}{0}{M\ell}{1}
		= \omega_1 \frac{R^{k+B/2}}{Q\ell}\sum_{d_{\lambda}=R} \zeta_{Q\ell}^{-\lambda(t+B/24)}  f\left(\pMatrix{R}{\lambda R}{0}{2^{\varepsilon}\ell} z\right)
\end{align}
where $\omega_{1} = \rho \, \chi(R) \xi^B$ is a $24N$-th root of unity. Since $f(z) = q^{\frac{B}{24}+n_{0}}+\cdots$ and $\frac{B}{24}+n_0<0$, the leading coefficient $K$ of \eqref{eq:sum1} is given by
\[
	K = \omega_{1}\frac{R^{k+B/2-1}}{2^\varepsilon\ell}\sum_{d_{\lambda}=R} \exp\left( 2\pi i \left( -\frac{\lambda t}{Q\ell} - \frac{\lambda B}{24Q\ell} + \frac{\lambda BR}{24 \cdot 2^{\varepsilon}\ell}+\frac{n_{0}\lambda R}{2^{\varepsilon} \ell}\right)\right).
\]
Let $\lambda_{0}$ be the smallest positive integer for which $d_{\lambda}=R$. Then the $\lambda$ for which $d_{\lambda}=R$ are of the form $\lambda_{0}+\delta R$ for $0\leq \delta \leq 2^{\varepsilon}\ell-1$. Write $\lambda=\lambda_0+\delta R$. Then
\begin{align*}
	-\frac{\lambda t}{Q\ell} - \frac{\lambda B}{24Q\ell} &+ \frac{\lambda BR}{24 \cdot 2^{\varepsilon}\ell} + \frac{n_{0}\lambda R}{2^{\varepsilon} \ell} \\
		&= \frac{\delta}{2^\varepsilon \ell} \left( B\frac{R^2-1}{24} + n_0 R^2 - t \right) + \underbrace{\frac{\lambda_0}{2^\varepsilon \ell} \left( B\frac{R^2-1}{24R} + n_0 R - \frac{t}{R} \right)}_{\text{independent of $\delta$}}.
\end{align*}
So, for some $24Q\ell N$-th root of unity $\omega_2$, we have
\begin{align*}
	K = \omega_{2}\frac{R^{k+B/2-1}}{2^\varepsilon\ell}\sum_{\delta=0}^{2^\varepsilon \ell-1} \exp\left( 2\pi i \frac{\delta}{2^{\varepsilon}\ell}\left( -t + B\frac{R^{2}-1}{24}+n_{0}R^{2}\right)\right).
\end{align*}
By assumption, 
\[
	t \equiv B\frac{R^{2}-1}{24} + n_{0}R^{2} \pmod{2^{\varepsilon}\ell},
\]
so $K = \omega_{2} R^{k+B/2-1}$.
\end{proof}

\bibliographystyle{plain}
\bibliography{bibliography.bib}

\end{document}